\documentclass[9pt,amsart]{article}
\usepackage{amsmath, amssymb, amsfonts, amsthm,amscd}
\usepackage{amssymb, graphicx}
\usepackage{graphicx}
\usepackage{young}
\usepackage{youngtab}
\usepackage{psfrag}
\usepackage[all, 2cell]{xy}
\usepackage{bm}
\usepackage{amssymb}
\usepackage{epsfig}
\usepackage{epstopdf}
\usepackage{mathrsfs}
\usepackage{amscd}
\usepackage{amsopn}
\usepackage{latexsym,amsmath, amsfonts, graphics, epsf, epic}
\usepackage{amssymb}
\usepackage[latin1]{inputenc}
\usepackage[english]{babel}
\usepackage{amsthm}
\usepackage{graphicx}
\usepackage{amsfonts}
\usepackage{amsmath}
\usepackage{latexsym}
\usepackage{amscd}
\usepackage{verbatim}

\newcommand{\mf}{\mathfrak}
\newcommand{\g}{\mf{g}}

\allowdisplaybreaks[1]

\newtheorem{teorem}{Theorem}[section]
\newtheorem{proposition}[teorem]{Proposition}
\newtheorem{lemma}[teorem]{Lemma}
\newtheorem{corollary}[teorem]{Corollary}
 
\theoremstyle{remark}
\newtheorem{remark}[teorem]{Remark}

\usepackage{color}
\definecolor{orange}{rgb}{1,.549,0}
 \definecolor{GreenYellow }{rgb}{ 0.15,   0.69, 0} 
\definecolor{Yellowone}{rgb}{ 0, 1., 0} 
\definecolor{Goldenrod }{rgb}{  0, 0.10, 0.84} 
\definecolor{Dandelion }{rgb}{ 0, 0.29, 0.84} 
\definecolor{Apricot }{rgb}{ 0, 0.32, 0.52} 
\definecolor{Peach }{rgb}{ 0, 0.50, 0.70} 
\definecolor{GreenYellow}{cmyk}{0.15,0,0.69,0}
\definecolor{RoyalPurple}{cmyk}{0.75,0.90,0,0}
\definecolor{Yellow}{cmyk}{0,0,1,0}
\definecolor{BlueViolet}{cmyk}{0.86,0.91,0,0.04}
\definecolor{Goldenrod}{cmyk}{0,0.10,0.84,0}
\definecolor{Periwinkle}{cmyk}{0.57,0.55,0,0}
\definecolor{Dandelion}{cmyk}{0,0.29,0.84,0}
\definecolor{CadetBlue}{cmyk}{0.62,0.57,0.23,0}
\definecolor{Apricot}{cmyk}{0,0.32,0.52,0}
\definecolor{CornflowerBlue}{cmyk}{0.65,0.13,0,0}
\definecolor{Peach}{cmyk}{0,0.50,0.70,0}
\definecolor{MidnightBlue}{cmyk}{0.98,0.13,0,0.43}
\definecolor{Melon}{cmyk}{0,0.46,0.5,0}
\definecolor{NavyBlue}{cmyk}{0.94,0.54,0,0}
\definecolor{YellowOrange}{cmyk}{0,0.42,1,0}
\definecolor{RoyalBlue}{cmyk}{1,0.50,0,0}
\definecolor{Orange}{cmyk}{0,0.61,0.87,0}
\definecolor{Blue}{cmyk}{1,1,0,0}
\definecolor{BurntOrange}{cmyk}{0,0.51,1,0}
\definecolor{Cerulean}{cmyk}{0.94,0.11,0,0}
\definecolor{Bittersweet}{cmyk}{0,0.75,1,0.24}
\definecolor{Cyan}{cmyk}{1,0,0,0}
\definecolor{RedOrange}{cmyk}{0,0.77,0.87,0}
\definecolor{ProcessBlue}{cmyk}{0.96,0,0,0}
\definecolor{Mahogany}{cmyk}{0,0.85,0.87,0.35}
\definecolor{SkyBlue}{cmyk}{0.62,0,0.12,0}
\definecolor{Maroon}{cmyk}{0,0.87,0.68,0.32}
\definecolor{Turquoise}{cmyk}{0.85,0,0.20,0}
\definecolor{BrickRed}{cmyk}{0,0.89,0.94,0.28}
\definecolor{TealBlue}{cmyk}{0.86,0,0.34,0.02}
\definecolor{Red}{cmyk}{0,1,1,0}
\definecolor{Aquamarine}{cmyk}{0.82,0,0.30,0}
\definecolor{OrangeRed}{cmyk}{0,1,0.50,0}
\definecolor{BlueGreen}{cmyk}{0.85,0,0.33,0}
\definecolor{RubineRed}{cmyk}{0,1,0.13,0}
\definecolor{Emerald}{cmyk}{1,0,0.50,0}
\definecolor{WildStrawberry}{cmyk}{0,0.96,0.39,0}
\definecolor{JungleGreen}{cmyk}{0.99,0,0.52,0}
\definecolor{Salmon}{cmyk}{0,0.53,0.38,0}
\definecolor{SeaGreen}{cmyk}{0.69,0,0.50,0}
\definecolor{CarnationPink}{cmyk}{0,0.63,0,0}
\definecolor{Green}{cmyk}{1,0,1,0}
\definecolor{Magenta}{cmyk}{0,1,0,0}
\definecolor{ForestGreen}{cmyk}{0.91,0,0.88,0.12}
\definecolor{VioletRed}{cmyk}{0,0.81,0,0}
\definecolor{PineGreen}{cmyk}{0.92,0,0.59,0.25}
\definecolor{Rhodamine}{cmyk}{0,0.82,0,0}
\definecolor{LimeGreen}{cmyk}{0.50,0,1,0}
\definecolor{Mulberry}{cmyk}{0.34,0.90,0,0.02}
\definecolor{YellowGreen}{cmyk}{0.44,0,0.74,0}
\definecolor{RedViolet}{cmyk}{0.07,0.90,0,0.34}
\definecolor{SpringGreen}{cmyk}{0.26,0,0.76,0}
\definecolor{Fuchsia}{cmyk}{0.47,0.91,0,0.08}
\definecolor{OliveGreen}{cmyk}{0.64,0,0.95,0.40}
\definecolor{Lavender}{cmyk}{0,0.48,0,0}
\definecolor{RawSienna}{cmyk}{0,0.72,1,0.45}
\definecolor{Thistle}{cmyk}{0.12,0.59,0,0}
\definecolor{Sepia}{cmyk}{0,0.83,1,0.70}
\definecolor{Orchid}{cmyk}{0.32,0.64,0,0}
\definecolor{Brown}{cmyk}{0,0.81,1,0.60}
\definecolor{DarkOrchid}{cmyk}{0.40,0.80,0.20,0}
\definecolor{Tan}{cmyk}{0.14,0.42,0.56,0}
\definecolor{Purple}{cmyk}{0.45,0.86,0,0}
\definecolor{Gray}{cmyk}{0,0,0,0.50}
\definecolor{Plum}{cmyk}{0.50,1,0,0}
\definecolor{Black}{cmyk}{0,0,0,1}
\definecolor{Violet}{cmyk}{0.79,0.88,0,0}
\definecolor{White}{cmyk}{0,0,0,0}

 \definecolor{rltred}{rgb}{0.75,0,0}
   \definecolor{rltgreen}{rgb}{0,0.5,0}
   \definecolor{oneblue}{rgb}{0,0,0.75}
   \definecolor{marron}{rgb}{0.64,0.16,0.16}
   \definecolor{forestgreen}{rgb}{0.13,0.54,0.13}

   \definecolor{purple}{rgb}{0.62,0.12,0.94}
   \definecolor{dockerblue}{rgb}{0.11,0.56,0.98}
   \definecolor{freeblue}{rgb}{0.25,0.41,0.88}
   \definecolor{myblue}{rgb}{0,0.2,0.4}

 \definecolor{Melon}{rgb}{ 0.46, 0.50, 0}

 \definecolor{Melone}{rgb}{ 0, 0.46, 0.50}

\begin{document}
\title{On certain modules of covariants in exterior algebras}
\author{Salvatore Dolce\footnote{Dipartimento di Matematica, Universit\`a La Sapienza di Roma, P.le Aldo Moro 5, 00185 Rome, Italy;\protect\\ E-mail address: dolce@mat.uniroma1.it;\protect\\ 2010 Mathematics Subject Classification 17B20;\protect}}

\maketitle

\begin{abstract}
We study the structure of the space of covariants $B:=\left(\bigwedge (\g/\mathfrak k)^*\otimes \g\right)^{\mathfrak k},$ for a certain class of infinitesimal symmetric spaces $(\g,\mathfrak k)$ such that the space of invariants $A:=\left(\bigwedge (\g/\mathfrak k)^*\right)^{\mathfrak k}$ is an exterior algebra $\wedge (x_1,...,x_r),$ with $r=rk(\g)-rk(\mathfrak k)$.

We prove that they are free modules over the subalgebra $A_{r-1}=\wedge (x_1,...,x_{r-1})$ of rank $4r$. In addition we will give an explicit basis of $B$.

As particular cases we will recover same classical results. In fact we will describe the structure of $\left(\bigwedge (M_n^{\pm})^*\otimes M_n\right)^G$, the space of the $G-$equivariant matrix valued alternating multilinear maps on the space of (skew-symmetric or symmetric with respect to a specific involution) matrices, where $G$ is the symplectic group or the odd orthogonal group. Furthermore we prove new polynomial trace identities.

\end{abstract}

\section*{Introduction} 
 
In this paper we study the ring of invariant skew symmetric multilinear functions on a linear representation of an algebraic group $G$, that is from a geometric viewpoint constant coefficient invariant differential forms. It is a classical fact that these forms give the cohomology of compact Lie groups and more generally compact symmetric spaces. (See for example \cite{Bo2}, \cite{Cartan}, \cite{Chevalley})

In the first part of the paper we are mostly concerned with classical groups. Let $G=GL(n)$ be the group of invertible $n\times n$ complex matrices. $G$ acts on the space $M_n$ of complex $n\times n$ matrices by conjugation. We first study, essentially following Procesi \cite{Pr1}, the algebra of invariant skew symmetric multilinear functions on $M_n$.
This turns out to be closely related to the theory of rings with Polynomial Identities.  

A fundamental role for our purposes will be played by the the space of $G$-equivariant multilinear alternating matrix valued maps on the $m-$tuples of matrices.

We will endow this space with a natural structure of algebra, by defining a suitable skew-symmetric product. It will be worthwhile to consider this algebra as a module on the algebra of invariants. 

Bresar, Procesi and Spenko \cite{BPS} have shown, in the case of the linear group $GL(n)$, that this algebra is a free module on a certain subalgebra of invariants, with a natural explicit basis. Section 1 is devoted to recollect these facts.

Most  of our paper is devoted to  extend this and related results to the case of certain symmetric spaces. Starting with $GL(n)$ (or better $SL(n)$) consider the orthogonal (respectively, when $n$ is even, symplectic) involution $\sigma$ and denote by $SO(n)$ (respectively $Sp(n)$) the special orthogonal (respectively symplectic) groups of elements in $SL(n)$ fixed by $\sigma$. 
$\sigma$ induces a linear involution on the space $M_n$ which decomposes as the direct sum $M_n^+\oplus M_n^-$ of the $+1$ and $-1$  eigenspaces.

In section 2 we give a precise description of the rings of invariant skew symmetric multilinear maps on $M_n^{\pm}$ and of the ring of $M_n$ valued invariant skew symmetric multilinear maps on $M_n^{\pm}$ at least in the case of the symplectic involution and of the orthogonal involution for $n$ odd.

We will prove that, similarly to the linear group, the space of the $G$-equivariant matrix valued alternating multilinear maps of the space of $m-$tuples of matrices (symmetric or skew-symmetric) is a free module on a certain subalgebra of invariants. Also in these cases, we exhibit an explicit basis of this module. 

Furthermore we recover some classical results analog to the Amitsur-Levitzki theorem. More precisely, for the symplectic group we recover a result of Rowen (see \cite{Ro}) which states that the skew-symmetric standard polynomial of degree $4n-2$ vanishes on $2n\times 2n$ matrices, symmetric with respect to the symplectic involution. 
For the odd orthogonal group we recover the corresponding result of Hutchinson (see \cite{Hu}) which states that the skew-symmetric standard polynomial of degree $4n$ is zero when restricted to skew-symmetric matrices $2n+1\times 2n+1$.

The case of the even orthogonal group deserves a study on its own. We will deal with it in a subsequent paper (\cite{Dolce}).

Inspired by the strategy  of \cite{PDP}, in which the result of Bresar et al. has been extended to a general result  for a simple Lie algebra $\mathfrak g$ concerning the space of invariant skew symmetric $\mathfrak g$ valued multilinear maps on $\mathfrak g$, we extend our results to a certain class of symmetric pairs (see section 3 for details). Indeed this will not be too hard since, apart from the symplectic and orthogonal (for odd $n$) involutions, the only other cases we need to consider are the pairs $(\mathfrak {so}(2n),\mathfrak {so}(2n-1))$ and $(\mathfrak e_6,\mathfrak f_4)$. The first case is easy to deal with since the corresponding symmetric space is an odd dimensional sphere. The second is treated combining ad hoc reasoning and computer aided computation (we use the software LiE). 
\subsection*{Acknowledgements}
I want to thank professors C.De Concini and C.Procesi for precious suggestions and advices. 

\section{General setting, ideas and goals}
\subsection{Antisymmetry}
In this work we will study multilinear antisymmetric identities. Let us first introduce the appropriate setting for this.
By the {\em antisymmetrizer} we mean the operator that sends a multilinear application $f(x_1,\ldots,x_h)$ into the antisymmetric application $ \sum_{\sigma\in  S_h} \epsilon_\sigma f(x_{\sigma(1)},\ldots,x_{\sigma(h)})$.

An important example for us is obtained applying the antisymmetrizer to the noncommutative monomial $x_1\cdots x_h$. We get the standard polynomial of degree $h$,  $$St_h(x_1,\ldots,x_h)=\sum_{\sigma\in  S_{h}}\epsilon_\sigma   x_{\sigma(1)} \dots x_{\sigma(h)}.$$
Up to a scalar multiple, this is the only multilinear antisymmetric noncommutative polynomial of degree $h$.

Let $R$ be any algebra (not necessarily associative) over a field $\mathbb{F}$, and let $V$ be a finite dimensional vector space  over $\mathbb F$. The set of  multilinear  antisymmetric functions from $V^k$ to $R$ can be identified in a natural way with $\bigwedge^kV^*\otimes R$.  Using the algebra structure of $R$ we have a wedge product of these functions; for $G\in \bigwedge^hV^*\otimes R,\ H\in \bigwedge^kV^*\otimes R$ we define
$$(G\wedge H)( v_1,\ldots,v_{h+k}):=\frac{1}{h!k!}\sum_{\sigma\in S_{h+k}}\epsilon_\sigma G (v_{\sigma(1)},\ldots,v_{\sigma(h)})H(v_{\sigma(h+1) },\ldots,v_{\sigma(h+k)}) $$
$$=\sum_{\sigma\in  S_{h+k}/ S_{h}\times  S_{k}}\epsilon_\sigma G (v_{\sigma(1)},\ldots,v_{\sigma(h)})H(v_{\sigma(h+1) },\ldots,v_{\sigma(h+k)}).$$
It is easy to show (see \cite{Pr1}) that $St_a\wedge St_b=St_{a+b} $.

With this multiplication 
the algebra of multilinear  antisymmetric functions from $V $ to $R$ is isomorphic to the tensor product algebra $\bigwedge V^*\otimes R$.  We shall denote by $\wedge$ the product in this algebra.

Assume now that $R$ is an associative algebra and $V\subset R$.  The inclusion map $X:V\to R$ is of course antisymmetric, since the symmetric group on one variable is trivial, hence $X\in \bigwedge V^*\otimes R$. By iterating the definition of wedge product we have the important fact:
\begin{proposition}
As a multilinear function, each power  $X^a:=X^{\wedge a}$  equals the standard polynomial $St_a$  computed in $V$.
\end{proposition}

We consider now an example which will be useful in the following.
We apply the previous considerations to $V=R=M_n(\mathbb C):=M_n$;  the group $G=PGL(n,\mathbb C)$ acts  on this space, and hence on functions, by conjugation and it is interesting to study the algebra of $G$--equivariant maps\begin{equation}
\label{an}B:=(\bigwedge M_n^*\otimes M_n)^G.
\end{equation} 

This among other topics is  discussed in \cite{BPS}. Here we provide a slightly different approach to the study of $B$ and we generalize it to the other classical groups.

Let us first consider the space of invariants $A:=\left(\bigwedge M_n^*\right)^G$. We see that $B$ is naturally a $A$-module. It follows from results on Chevalley trasgression \cite{Chevalley} and Dynkyn \cite{Dynkyn} that $A$ is the exterior algebra in the elements:
$$T_h:=Tr(St_{2h+1}(x_1,...,x_{2n+1})),\ \ \ \ \ \ i=0,...,n-1.$$
We remark that we use only traces of the standard polynomials of odd degree since, as it is well-known, $Tr(St_{2h}(x_1,...,x_{2h})) = 0$ for every $h$, see \cite{Rosset}.

Consider now the graded super algebra $A[t]$, with deg $t=1$. Notice that for each $i$,\\ $ T_it=-tT_i.$ Define the graded algebra homomorphism 
$\pi:A[t]\to B$ by $$\sum_ja_jt^j\mapsto \sum_ja_j\wedge X^j ,$$
$a_j\in A$.  By classical results (see for example \cite{Pr2}) we know that $\pi$ is surjective. Furthermore by \cite{BPS},\cite{Pr1} we have:
\begin{teorem}\label{Procesicov}
\begin{enumerate}
\item The algebra $B$ is a free module on the subalgebra $A_{n-1}\subset A$ generated by the elements $T_i$, $i=0,...,n-2$, with basis $1,...,X^{2n-1}$.
\item The kernel of the canonical homomorphim $\pi:A[t]\rightarrow B$ is the principal ideal generated by 
$$\sum_{i=0}^{ n-1}t^{2i} T_{n- i-1}-nt^{2n-1}.$$
\end{enumerate}
\end{teorem}
The proof of this theorem uses two main results. The first is a result of Kostant (see \cite{Ko3}) on the dimension of $B$ and the second is the Amitsur-Levitzki's Theorem (see \cite{AL}).

An interesting remark is that, assuming theorem \ref{Procesicov}, one could deduce the Amitsur-Levitzki's identity.

In fact multiplying on the left by $t$ the identity above we obtain $nt^{2n}-\sum t^{2i+1}T_{n-i-1}$, while multiplying on the right by $t$ we obtain $nt^{2n}-\sum t^{2i}T_{n-i-1}t$. So $t^{2n}\in Ker(\pi)$ (by the fact that in $A[t]$ we have $T_it=-tT_i$). It is easy to see that, in our setting, $X^{2n}=0$ is equivalent to the Amitsur-Levitzki's identity.

We will use this trick to deduce other interesting identities in the case of symmetric or skew-symmetric matrices.
 \subsection{Covariants in exterior algebras of Lie algebras with involution}

In the following we will use previous ideas to show similar results for a certain class of Lie algebras with involution.

So let $\g$ be a complex finite dimensional semisimple Lie algebras and let $\sigma:\g\rightarrow \g$ be an indecomposable involution. We denote by $\g=\mathfrak{k}\oplus\mathfrak{p}$ its Cartan decomposition, where $\mathfrak{k}$ is the fixed point set of $\sigma$. By classical results (see for example \cite{Helgason}) we know that $\g$ is either simple or a sum of two simple ideals switched by the flip involution.

The main object of our interest will be the space $\left(\bigwedge\mathfrak{p}^*\otimes \g\right)^{\mathfrak{k}}$ for pairs $(\g,\sigma)$ such that the algebra of invariants $\left(\bigwedge\mathfrak{p}^*\right)^{\mathfrak{k}}$ is an exterior algebra.

For the case $\g$ non simple, it is a celebrated theorem of Hopf, Samelson and Koszul that the algebra\\ $\left(\bigwedge\mathfrak{p}^*\right)^{\mathfrak{k}}=\left(\bigwedge\mathfrak{k}^*\right)^{\mathfrak{k}}$ is always an exterior algebra over primitive generators $P_i$ of degree $2m_i+1$, where the integers $m_i$, with $m_1\leq...\leq m_r$ are the exponents of $\Delta$ and $r$ is the rank of $\mathfrak{k}$.

One of the main result of \cite{PDP} is the following theorem:
\begin{teorem}{\label{inspiration}}
The algebra $B:=\left(\bigwedge\mathfrak{k}^*\otimes \mathfrak{k}\right)^\mathfrak{k}$ is a free module of rank $2r$ on the subalgebra $A_{r-1}\subset A:=\left(\bigwedge\mathfrak{k}^*\right)^\mathfrak{k}$ generated by $P_1,...P_{r-1}$.
\end{teorem}
For the case $\g$ simple we have that from classical results (see \cite{Helgason}, \cite{Bo2}, \cite{Tak}) the only four pairs $(\g,\sigma)$ such that $\left(\bigwedge\mathfrak{p}^*\right)^{\mathfrak{k}}$ is an exterior algebra are:
\begin{enumerate}
\item[1.] $(\mathfrak{sl}(2n),-s)$, where $s$ is the symplectic transposition and $\mathfrak{k}=\mathfrak{sp}(2n)$,
\item[2.] $(\mathfrak{sl}(2n+1),-t)$, where $t$ is the usual transposition and $\mathfrak{k}=\mathfrak{so}(2n+1)$, 
\item[3.] $(\mathfrak{so}(2n),\sigma_1)$, where $\sigma_1$ is an involution such that $\mathfrak{k}=\mathfrak{so}(2n-1)$,
\item[4.] $(\mathfrak{e}_6,\sigma_2)$, where $\sigma_2$ is an involution such that $\mathfrak{k}=\mathfrak{f}_4$. 
\end{enumerate}

In this paper we study the covariants in these cases so that our results plus Theorem \ref{inspiration} will give the following:

\begin{teorem}\label{mainthe}
Let $\g$ be a complex finite dimensional semisimple Lie algebra. Let $\sigma:\g\rightarrow\g$ be an indecomposable involution such that $\left(\bigwedge \mathfrak{p}^*\right)^{\mathfrak{k}}$ is an exterior algebra of type $\wedge (x_1,...,x_r)$, where the $x_i$'s are ordered by their degree and $r$ is such that $r:=rk(\mathfrak{g})-rk(\mathfrak{k})$. We have that the algebra $B:=\left(\bigwedge \mathfrak{p}^*\otimes \mathfrak{g}\right)^\mathfrak{k}$ is a free module of rank $4r$ on the subalgebra of $\left(\bigwedge \mathfrak{p}^*\right)^\mathfrak{k}$ generated by the elements $x_1,...,x_{r-1}$. 
\end{teorem}

\section{Invariant theory in exterior algebras for the symplectic and orthogonal groups}
\subsection{Symplectic case}

We denote, for all $k\in \mathbb{N}$, by $1_k$ the identity matrix of order $k$ and by $A^t$ the usual transposition of a matrix $A$.

We consider the skew-symmetric matrix  $J:=\begin{pmatrix}0&1_n\\ -1_n&0
\end{pmatrix}$.

On the space $V:=\mathbb C^{2n}$  we have the skew-symmetric form  $(u,v):=u^tJv$, where $u$ and $v$ are column vectors in $\mathbb C^{2n}$.

The symplectic transposition, $A\mapsto A^s$, is  defined by
$$ A^s:=- JA^tJ.$$
Explicitly we have that, if $M=\begin{pmatrix}A&B\\ C&D
\end{pmatrix}$ with $A,B,C,D$ $n\times n$,  its symplectic transpose is
\begin{equation}
M^s=\begin{pmatrix}D^t&-B^t\\ -C^t&A^t
\end{pmatrix}.
\end{equation}

Let $M_{2n}^+=\{A\,|\, A^s=A\}$ be the space of symmetric matrices with respect to the symplectic transposition. Notice that the map  $A\mapsto AJ$ gives a linear isomorphism of the space $M_{2n}^+$ onto the space $\bigwedge^2V$ of skew-symmetric matrices (with respect to the usual transposition). The group $GL(2n,\mathbb C)=GL(V)$ acts on $\bigwedge^2V$  by $ X\circ A:=XAX^t$, for $X\in GL(V)$   and  $A\in \bigwedge^2V$.

Let $G\subset GL(V)$ be the symplectic group. By definition $G$ is the group of transformations preserving the symplectic form, that is
$$G:=\{X\in GL(V)\,|\,  J= X^tJX\}. $$ 
This is equivalent to require that $ X^{-1}=  X^s$. So we can see the symplectic group as the fixed point set of the involution of $GL(V)$ given by $X\mapsto (X^s)^{-1}$. 

The Lie algebra $M_{2n}^-$ of $G$ is the space of the skew-symmetric matrices, with respect to the symplectic form, $M_{2n}^- := \{A |\ A^s =-A\}$. The map $A\mapsto AJ$ gives a linear isomorphism of $M_{2n}^-$  onto the space $S^2(V)$ of symmetric matrices (with respect to the usual transposition). Also in this case $GL(V)$ acts on $S^2(V)$ by 
$ X\circ A:=XAX^t$, for $X\in GL(V)$   and  $A\in S^2V$.

The conjugation action of $G$  commutes with the map $A\mapsto AJ$. In fact if $X\in G$  we have $XJX^{-1}=J$, so we can state that 
\begin{proposition}\label{equivrep}
  The action of the symplectic group $G$ on $M_{2n}^+$ (resp. on $M_{2n}^-$) can be identified with the restriction to $G$ of the usual action of the linear group $GL(V)$ on $\bigwedge^2V$ (resp. on $S^2(V)$).\end{proposition}
Remark that with respect to the action of $G$, while the Lie algebra $L:=M_{2n}^-\simeq S^2(V)$ is irreducible,  $M_{2n}^+$ is not irreducible. Indeed it decomposes as the direct sum of the one dimensional space of scalar matrices and of the space  $P_0$ of traceless matrices.  Under the isomorphism with $\bigwedge^2V$ the space of scalar matrices maps to the space spanned by  $J$.
\smallskip

\subsection{Invariants of the representation $\bigwedge (M_{2n}^{\pm})^*\otimes M_{2n}$}
\subsubsection{Dimension}

Here and below we index the irreducible representations of $GL(V)$ by Young diagrams with at most $n$ columns (the row of length $k$ corresponds to $\bigwedge ^kV$). The irreducible module corresponding to the diagram $\lambda$ will be denoted by $S_{\lambda}(V).$

Let us start from the case $M_{2n}^+\simeq \bigwedge^2 V$. By the {\em Plethysm formulas} (see \cite{MD}) we know how to decompose $\bigwedge[\bigwedge^2V]$ with respect to the action of the linear group $GL(V)$.

We have that $\bigwedge[\bigwedge^2V]$ is the direct sum of the irreducible representations $$H_{a_1,a_2,\ldots,a_k}^-(V)=H_{\underline a}^-(V):=S_{\lambda(\underline a)}(V),$$
 with $2n>a_1>a_2  \ldots >a_k> 0$. Notice that $H^-_\emptyset$ is the trivial one dimensional representation.

 The Young diagram  $\lambda(\underline a)$ is built by nesting   the  {\em hook} diagrams $h_{a_i}$  whose column is of length $a_i$ and whose row is of the length $a_i+1$,.
As an example, the diagram $\lambda(4,3,1)$ is
$$
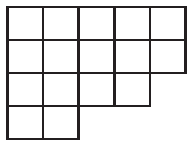
\begin{Young}
      &&&&\cr
     &&&& \cr
     &&& \cr
     &\cr
\end{Young}$$

Using this, we can compute the dimension of the invariants for the symplectic group $G$ in $\bigwedge M_{2n}^+$.
We know that, for each diagram $\lambda$, dim$(S_\lambda(V))^{G})\leq 1$ and dim$ (S_{\lambda}(V)^{G})=1$ if and only if every row of $\lambda$ is even, (see \cite{Pr2}).

The rows of a representation $H_{\underline a}^-(V)$  are even if and only if  $\underline a=b_1,b_1-1,b_2,b_2-1,\ldots ,b_s,b_s-1,\ldots$, or
$\underline a=b_1,b_1-1,b_2,b_2-1,\ldots ,b_{s},1$ with the $b_j$'s odd and $b_s>1$. 
So we have to compute the number of decreasing sequences of odd numbers smaller than $2n$. This is the same to compute the number of decreasing sequences of  numbers taken from $1,...,n$, that is the number of subsets of $\{1,\ldots ,n\}.$ Thus
\begin{proposition}\label{pep}
The dimension of the space of invariants $(\bigwedge M_{2n}^+)^G$ is $2^n$.
\end{proposition}

The case $M_{2n}^-$ is quite similar. We set $L:=M_{2n}^-\simeq S^2V$.

\begin{remark}
Remark that $L$ is the Lie algebra of type $C_n$ and in this part of the work we will recover general results for covariants in the exterior algebra of a Lie algebra $\mathfrak{g}$ as in $\cite{PDP}$.  
\end{remark}

By the {\em Plethism formulas} (see \cite{MD}) we know how to decompose $\bigwedge[S^2V]$ with respect to the action of the linear group $GL(V)$.

We have that $\bigwedge[S^2V]$ is the direct sum of the irreducible representations
\begin{equation}\label{decsim}
H_{a_1,a_2,\ldots,a_k}^+(V)=H_{\underline a}^+(V)=S_{\lambda(\underline a)}(V),\end{equation}
 with $2n>a_1>a_2  \ldots >a_k\geq 0$.

This time however the Young diagram  $\lambda(\underline a)$ is built by nesting   the  {\em hook} diagrams $h_{a_i}$  whose column is of length $a_i+2$ and whose row is of the length $a_i+1$.

Now (analogous to the previous case) we want to compute the dimension of the invariants for the symplectic group.

We have that the rows of a representation $H_{\underline a}^+(V)$ are even if and only if the sequence is of the type $a,a-1,b,b-1,c,c-1,\ldots$, with $a,b,c,...$ odd. 
So as before we can state: 
\begin{proposition}
The dimension of the space of invariants $(\bigwedge M_{2n}^-)^G$ is $2^{n}$.
\end{proposition}

 We now pass to determine the dimension of the space  $B:=(\bigwedge M_{2n}^{\pm}\otimes M_{2n})^G$.
 Before we proceed, let us make a few remarks. We know that to study an isotypic component relative to an irreducible representation $N$ in a representation $M$ of a reductive group $G$, we need to study the space $\hom_G(N,M)$ that we can identify to $(N^*\otimes M)^G$. The dimension of this space will be the number of copies of $N$ in $M$.  
 
 We decompose the space of matrices $M_{2n}=M_{2n}^+\oplus M_{2n}^-=\mathbb C\oplus P_0 \oplus M_{2n}^-$.
 Thus we have $$\left(\bigwedge M_{2n}^{\pm}\otimes M_{2n}\right)^G=\left(\bigwedge M_{2n}^{\pm}\otimes \mathbb C\right)^G\oplus \left(\bigwedge M_{2n}^{\pm}\otimes P_0\right)^G\oplus \left(\bigwedge M_{2n}^{\pm}\otimes M_{2n}^-\right)^G$$
 The space $(\bigwedge M_{2n}^{\pm}\otimes \mathbb C)^G$ is the space of invariants whose dimension we have already computed. Since $P_0$ and $M_{2n}^-$ are both irreducible and self dual, our computation will give us information on the isotypic components relative to these two representations.

 Now identify, as a representation of $G$,   $M_{2n}\simeq V\otimes V$.  We thus have to study the $G$-invariants in  
 $$\bigwedge (M_{2n}^{\pm})\otimes V\otimes V=\oplus_{\underline a} H_{\underline a}^{\mp}(V)\otimes V\otimes V.$$
 
 By Pieri's formulas we know how to decompose $S_\lambda(V)\otimes V\otimes V$. First we decompose $S_\lambda(V)\otimes V$. We have that $S_\lambda(V)\otimes V=\oplus S_{\lambda_i}(V) $ where $\lambda_i$ runs along the diagrams whose first row is of length at most $2n$ and which are obtained  from  $\lambda$ by adding one box.
So, iterating the same process, to compute the dimension of the invariants we have to study when, adding two squares, starting from one of the diagrams $H^-$ we obtain a diagram with even rows.
\begin{teorem}\label{ladim}
The dimension of $B$ is $(2n-1)2^n$ (resp. $(2n)2^n$) in the case $M_{2n}^+$ (resp. $M_{2n}^-$).
\end{teorem}
\begin{proof}
We prove only the case $M_{2n}^+$. The proof of the case $M_{2n}^-$ is similar.

We proceed by induction on $n$. The case $n=1$ is trivial so we assume $n>1$.

We have to compute the sequences of type:
$$2n>a_1>a_2>...>a_k>0,$$
such that when we add two squares to the corresponding diagram  we obtain a diagram with even rows.

We can have four different situations.\\
Let us start by considering the case $a_1<2(n-1)=2n-2$. By induction the number of diagrams with even rows whose first row has at most $2(n-1)$ boxes is $(2(n-1)-1)2^{n-1}$. However if $a_1=2n-3$ we can add two squares to the first row so that its length is ($a_1+3=2n$).  Thus we have a contribution from each sequence $\underline a=2n-3>2n-4>b_2>b_2-1>\ldots >b_s-1$ or $\underline a=2n-3>2n-4>b_2>b_2-1>\ldots >b_s-1>1$ with the $b_j$'s odd and $b_s>1$. Reasoning exactly as in the proof of Proposition  \ref{pep} we see that there are exactly  $2^{n-2}$ such diagrams. 

The second case is 
$$a_1=2n-1\ \mbox{and}\ a_2=2n-2,$$
In this case we have two possibilities. We can  add two squares at the bottom
$$
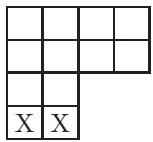
\begin{Young}
      &&&\cr
     &&&\cr
& \cr
     X&X \cr
\end{Young}$$
and we get a contribution from each sequence $\underline a=2n-3>2n-4>b_2>b_2-1>\ldots >b_s-1$ or $\underline a=2n-1>2n-2>b_2>b_2-1>\ldots >b_s-1>1$ with the $b_j$'s odd and $b_s>1$ and there are
 $2^{n-1}$ such sequences. 
 
Otherwise we add  the two squares to the diagram associated to the sequence $2(n-1)>a_3>a_4>...>a_k>0,$ so that, by  induction on $n$, we get $(2(n-1)-1)2^{n-1}$ contributions.

The third case is: 
$$a_1=2n-1\ \mbox{and}\ a_2<2n-2,$$
in this case the only way to add two squares is to add one on the second row and one to the second column
$$\begin{Young}
      &&&\cr
     &&&X\cr
& \cr
     &X \cr
\end{Young}$$
so we would get  $2^{n-2}$ contributions, but we can exchange the order in which we add the squares, so we get  $2\cdot 2^{n-2}$ contributions.

In the  case we  have $a_1=2n-2$, so that in order to get a diagram with even rows, we need to add a box to the first row. If $a_2<2n-4$, there are at least two rows of length $1$ so that we cannot obtain a diagram with even rows by adding a single box.
If $a_2=2n-3$ we need to add the second square to the second row 
$$\begin{Young}
      &&&&X\cr
     &&&&X\cr
& \cr
     & \cr
\end{Young}$$ 
and one easily sees that we get  $2^{n-2}$ contributions. 

If $a_2=2n-4$ we need  to add a box to the second column:
$$\begin{Young}
      &&&&&X\cr
     &&&\cr
& \cr
     & \cr
&X\cr
\end{Young}$$ 
and in this case we  get  $2\cdot 2^{n-2}$\ contributions, since we can exchange the order in which the boxes are inserted.

Finally  adding  all the contributions we obtain 
$$(2(n-1)-1)2^{n-1}+2^{n-2}+2^{n-1}+(2(n-1)-1)2^{n-1}+2^{n-1}+2^{n-2}+2^{n-1}=(2n-1)\cdot 2^n.$$
This is our claim.
\end{proof}

\subsubsection{Structure of algebras of invariants of the representation $\bigwedge (M_{2n}^+)^*\otimes M_{2n}$}

Let $M_{2n}$ be, as in the previous section, the space of complex matrices $2n\times 2n$ and let $M_{2n}^+:=\{x\in M_{2n}|x=x^s\}$ be the symmetric part with respect to the symplectic involution.  
Since both $ \bigwedge^*(M_{2n}^+)^*$ and $\bigwedge^*(M_{2n}^+)^*\otimes M_{2n}$ have a structure of graded associative algebras, also the algebras $$A:=\left(\bigwedge^*(M_{2n}^+)^*\right)^G\ \  \text{and}\ \  B:=\left(\bigwedge^*(M_{2n}^+)^*\otimes M_{2n}\right)^G$$ have a natural structure of graded associative algebras. Furthermore $B$ is clearly an $A$-module. These are the structures we want to investigate.

As an algebra $A$ is generated by the element $Tr(St_{2h+1}(x_1,...,x_{2h+1}))$. Indeed  by  classical invariant theory, we know that the polynomial invariant functions on the space $M_{2n}$ are generated by traces of the monomials in the variables $x_i,x_i^s$ (FFT for matrices, see \cite{Pr2}); furthermore since $M_{2n}^+$ is a  $G-$stable subspace of $M_{2n}$, we have that every  $G-$invariant polynomial function on $M_{2n}^+$ is the restriction of a $G-$invariant polynomial on $M_{2n}$ (since $G$ is a linear reductive algebraic group). So by multilinearizing and alternating we have our claim (recall that $Tr(St_{2h}(x_1,...,x_{2n}))=0$).

Since $M_{2n}=M_{2n}^-\oplus M_{2n}^+$ we have that $\bigwedge(M_{2n}^+)^*\otimes M_{2n}= \left(\bigwedge(M_{2n}^+)^*\otimes M_{2n}^-\right)\oplus\left(\bigwedge(M_{2n}^+)^*\otimes M_{2n}^+\right)$ and, passing to the invariants, $B=B^+\oplus B^-$, with $B^{\mp}=(\bigwedge(M_{2n}^+)^*\otimes M_{2n}^{\pm})^G$.

Further, we can define the element
  \begin{equation}
 X\in B_1\end{equation}
by $X(x)=x$ for any  $x\in M_{2n}^+$ (recall that $M_{2n}^+\subset M_{2n}$).
 The following simple Proposition gives some crucial properties of $A$ and $B$:

\begin{proposition}\label{pippis}We have:
\begin{enumerate}\item
The element $Tr(St_{4k+3}(x_1,...,x_{4k+3}))$ is zero for each $k\geq 0$. 
\item   $X^{k}\in B^-$ if and only if $k\equiv 0,1$ modulo 4\item $X^{k} \in B^+$if and only if $k\equiv 2,3$ modulo 4.\end{enumerate}
\end{proposition}
\begin{proof} We prove part $1$, the others being similar.

By the explicit form of the symplectic transposition (1) we have $Tr(x)=Tr(x^s)$, so
$$Tr(St_{2h+1}(x_1,...,x_{2h+1}))=Tr(St_{2h+1}(x_1,...,x_{2h+1})^s),$$ 
but
$$\begin{aligned}St_{2h+1}(x_1,...,x_{2h+1})^s&=\sum_{\sigma\in S_{2h+1}}\epsilon_{\sigma}(x_{\sigma(1)}\cdots x_{\sigma(2h+1)})^s\\ &=\sum_{\sigma\in S_{2h+1}}\epsilon_{\sigma}x_{\sigma(2h+1)}\cdots x_{\sigma(1)},\end{aligned}$$
and if we consider the permutation $\eta:(1,...,2h+1)\rightarrow (2h+1,...,1)$ we see that it has the same sign of the parity of $h$. So everything follows.
\end{proof}

We deduce  the following well know result (see for example \cite{Bo}):
\begin{teorem}\label{ciccio}
The algebra $A$ is the exterior algebra, of dimension $2^n$, in the elements\\ $T_0,T_1,...,T_{n-1}$, where $T_{h}:=Tr(S_{4h+1}(x_1,...,x_{4h+1}))\in A_{4h+1}$.
\end{teorem}
\begin{proof} By the previous Proposition, $A$ is generated by the elements $Tr(St_{4h+1}(x_1,...,x_{4h+1}))$. By the Amitzur-Levitzki theorem, $St_{r}(x_1,...,x_{r})=0$ for $r\geq 4n$. It follows that $A$ is generated by  $T_0,T_1,...,T_{n-1}$, so it is the quotient of an exterior algebra on $n$ generators. Since dim\ $A=2^n$, our claim follows.\end{proof}

Notice that we can define the trace function $Tr:  \bigwedge (M_{2n}^+)^*\otimes M_{2n}\to \bigwedge M_{2n}^+$ by extending the invariant function on matrices. By equivariance, on $B$ the trace function takes values in $A$.  In particular notice that $Tr(X^h)(x_1,\ldots x_k)=Tr(St_h(x_1,\ldots x_k))$. 
 
\begin{teorem}\label{tony}We have:
\begin{enumerate}
 \item  As a algebra, $B$ is generated by $A$ and the element $X$.
 \item   $B$ is a free module on the 
  the exterior algebra $A_{n-1}\subset A$ generated by $T_0,...,T_{n-2}$, with basis $1,X,...,X^{4n-3}$.
 \end{enumerate}

  \end{teorem}
\begin{proof} 1) follows from classical invariant theory (see \cite{Pr3},\cite{Pr2},\cite{BPS}).

2)  By  Proposition \ref{ladim}, it suffices to see that the elements  $1,X,...,X^{4n-3}$ are linearly independent over $A_{n-1}$. So, let
$$\sum_{h=0}^{4n-3}P_j\wedge X^j=0,$$ $P_h\in A_{n-1}$ for each $h=0,\ldots , 4n-3$. Assume by contradiction that not all $P_h$'s are $0$. Let $j$ be the minimum such that $P_j\neq 0$. Multiply by $X^{4n-3-j}$ and take traces. We get
$P_j\wedge T_{n-1}=0$, since for $h\geq 4n$, $X^{h}=0$,    $Tr(X^{4n-2})=Tr(X^{4n-1})=0$  by Proposition \ref{pippis} and the fact that $4n-2$ is even.
$P_j\in A_{n-1}$, so $P_j\wedge T_{n-1}=0$ if and only if $P_j=0$, a contradiction.
\end{proof}
Consider now the graded superalgebra $A[t]$, with deg $t=1$. Notice that for each $i$,\\ $ T_it=-tT_i.$ Define the graded algebra homomorphism 
$\pi:A[t]\to B$ by $$\sum_ja_jt^j\mapsto \sum_ja_jX^j ,$$
$a_j\in A$.  By the previous theorem we see that $\pi$ is surjective. Let us describe its kernel.
\begin{teorem}\label{tony2}  The Kernel $I$ of the homomorphism $\pi$ is the principal ideal generated by the element
 $$ nt^{4n-3}-\sum_{i=0}^{n-1}\frac{1}{2} T_{n-i-1}t^{4i}.$$
 \end{teorem}\begin{proof}
Let us first start by showing that the element $nt^{4n-3}-\sum_{i=0}^{n-1}\frac{1}{2} T_{n-i-1}t^{4i}$ lies in $I$. By the second part of Theorem \ref{ciccio}, we necessarily have a homogeneous relation
$$T_{n-1}=\sum_{h=0}^{4n-3}P_h\wedge X^{h}$$
with $P_h\in A_{n-1}$ of degree $4n-3-h$. Let us compute $P_{4n-3}$. Notice that $Tr(\sum_{h=0}^{4n-4}P_h\wedge X^{h})\in A_{n-1}$, while $Tr(T_{n-1})=2nT_{n-1}$ It follows that
$(2n-P_{4n-3})\wedge T_{n-1}\in A_{n-1}$ namely $P_{4n-3}=2n$. Assume now $0\leq j\leq 4n-4$ and multiply by $X^{4n-3-j}$. Taking the trace and reasoning as in the proof of Theorem \ref{tony}, we
get that $$T_{n-1}\wedge (Tr(X^{4n-3-j})+(-1)^{j}P_j)\in A_{n-1}$$
If $j=0$ we get $P_0=0$. Assume $j>0$. This implies 
$Tr(X^{4n-3-j})+(-1)^{j}P_j=0$ that is
$$P_j= \begin{cases} 0 \ \ \text{if}\ j=2h+1, 4h+2\\ T_{h}\  \ \ \text{if}\ i=4h+1\end{cases}$$
Thus dividing by 2 we obtain\begin{equation}\label{laequa}\frac{1}{2}T_{n-1}=nX^{4n-3}-\sum_{j=1}^{n-1}\frac{1}{2}T_{n-j-1}\wedge X^{4j}\end{equation}
which is our relation.
Denote by $J$ the ideal generated by 
$$ nt^{4n-3}-\sum_{i=0}^{n-1}\frac{1}{2} T_{n-i-1}t^{4i}.$$
We have seen that  $J\subset I$.

Now we show that $t^{2n-2}\in J$. To see this let us write the relation \eqref{laequa} as
$$nt^{2n-3}-\sum_{h=0}^{n-1}T_{n-h-1}t^{4h}.$$
Remark that multiplying on the right by $t$ we get
$$nt^{2n-2}-\sum_{h=0}^{n-1}T_{n-h-1}t^{4h+1}\in J.$$ 
If we multiply by $t$ on the left we get
$$nt^{2n-2}-\sum_{h=0}^{n-1}tT_{n-h-1}t^{4h}=nt^{2n-2}+\sum_{h=0}^{n-1}T_{n-h-1}t^{4h+1}.$$ 
Adding the two relation we thus obtain 
\begin{equation}\label{hutch}
2nt^{2n-2}\in J.
\end{equation}

Now consider $A/J$. Using \eqref{hutch}, it is clear that the image of $1,t,\ldots t^{4n-3}$ span  $A/J$ as a $A$ module. However by \eqref{laequa} we deduce that
$$T_{n-1}=2nt^{4n-3}-\sum_{i=1}^{n-1} T_{n-i-1}t^{4i},$$
so that, substituting, the same elements span $A/J$ as a $A_{n-1}$ module. It follows that dim $A/J\leq (4n-2)2^{n-1}=(2n-1)2^n$. On the other hand we know by Proposition \ref{ladim} that
dim $A/I$= dim $B=(2n-1)2^n$. We deduce that $I=J$.
\end{proof}

\begin{remark} The fact that $X^{2n-2}=0$ means that the standard polynomial $St_{4n-2}(x_1,...,x_{4n-2})$ is identically $0$ for $x_1,\ldots ,x_{4n-2}\in M_{2n}^+$. This is a result of Rowen \cite{Ro}.\end{remark}

Finally using Proposition \ref{pippis} we deduce,
\begin{corollary}
1) The elements $1,X,X^4,X^5,..., X^{4n-3}$ are a basis of $B^-$ as a free module  over the exterior algebra $A_{n-1}$.

2) The elements $X^2,X^3,X^6,X^7,...,X^{4n-5}$ are a basis $B^+$ as a free module over the exterior algebra $A_{n-1}$.
\end{corollary}

\subsubsection{Structure of invariants of the representation $\bigwedge (M_{2n}^-)^*\otimes M_{2n}$}
By the classical invariant theory we have that polynomial functions $G$-equivariant from $m$-copies of $L$ to $M_{2n}$, are an associative algebra generated by {\em coordinates} $Y_i$ and by the traces of monomials in $Y_i$. So we have that $[\bigwedge L^*\otimes M_{2n}]^G$  is generated by the element $Y,\ Y^s=-Y$  and $Tr(Y^i)$. Further we can decompose in a natural way $B=B^+\oplus B^-$ as

$$B:=[\bigwedge  L^*\otimes  M_{2n}]^G=[\bigwedge  L^*\otimes  L]^G\oplus [\bigwedge L^*\otimes M_{2n}^+]^G.$$ 
We have an analog of the Proposition \ref{pippis}:
\begin{proposition}\label{pippis2}
\begin{enumerate}\item
The element $Tr(St_{4k+1}(y_1,...,y_{4k+1}))$ is zero for each $k\geq 0$. 
\item   $Y^{k}\in B^-$ if and only if $k\equiv 0,3$ modulo 4\item $Y^{k} \in B^+$ if and only if $k\equiv 1,2$ modulo 4.\end{enumerate}

\end{proposition}
\begin{proof}
 As in the Proposition \ref{pippis} we just prove the first fact.

We recall that
$$Tr(St_{2h+1}(y_1,...,y_{2h+1}))=Tr(St_{2h+1}(y_1,...,y_{2h+1})^s),$$
but
$$\begin{aligned}St_{2h+1}(y_1,...,y_{2h+1})^s&=\sum_{\sigma\in S_{2h+1}}\epsilon_{\sigma}(y_{\sigma(1)}\cdots y_{\sigma(2h+1)})^s\\ &=\sum_{\sigma\in S_{2h+1}}(-1)^{2h+1}\epsilon_{\sigma}y_{\sigma(2h+1)}\cdots y_{\sigma(1)},\end{aligned}$$
and if we consider the permutation $\eta:(1,...,2h+1)\rightarrow (2h+1,...,1)$ we see that it has the same sign of the parity of $h$. So everything follows.
\end{proof}

We set 
\begin{equation}
T_h:=Tr(St_{4h+3}(Y_1,...,Y_{4h+3}))
\end{equation}
we have, by the computation of the dimension and by the Amitsur-Levitzki's Theorem, that
\begin{lemma}
The algebra $A:=(\bigwedge M_{2n}^-)^G$ is the exterior algebra in the elements  $T_0,...,T_{n-1}$.
\end{lemma}
So we can describe the structure of covariants:
\begin{teorem}\label{tony3}
\begin{enumerate}
\item The algebra $B$ is a free module on the subalgebra $A_{n-1}\subset A$ generated by the elements $T_i$, $i=0,...,n-2$, with basis $1,...,Y^{4n-1}$.
\item With analogous notations to the previous section we have that the kernel of the canonical homomorphim $\pi:A[t]\rightarrow B$ is the principal ideal generated by 
$$\sum_{i=0}^{ n-1}t^{4i} \wedge  T_{n- i-1}-2nt^{4n-1}.$$
\end{enumerate}

\begin{proof}
By computing dimensions we only need a formula to compute the multiplication by the element $T_{n-1}=Tr(Y^{4(n-1)+3})=Tr(Y^{4 n-1 })$, but in this case we can use the general formulas for the matrices (see [2]).

We start from the universal formula for $M_{2n}$:

\begin{equation}\label{genmat}
1\wedge T(4n-1)=T(4n-1)=-\sum_{i=1}^{2n-1}Z^{2i} \wedge  T(2(2n-i)-1) +2nZ^{4n-1},\end{equation}
where $Z:M_{2n}\rightarrow M_{2n}$ is the identity map and $T(h):=Tr(St_{h}(x_1,...,x_h))$.\\
In our case $Z\leadsto Y$ and we have (by \ref{pippis2}) that $T(2(2n-i)-1)=0$  unless that \\ $-2i-1=3\mbox{ mod}(4)$. So $i$ is even. We deduce:

\begin{equation}
\label{liden}T_{n-1}=-\sum_{i=1}^{ n-1}Y^{4i} \wedge  T_{n- i-1} +2nY^{4n-1},
\end{equation}  
so
\begin{equation}
\label{liden2}Y^j\wedge T_{n-1}=-\sum_{i=1}^{ n-\left[\frac j 4\right]}Y^{4i+j} \wedge  T_{n-i-1}.
\end{equation}
\end{proof}
\end{teorem}

So by this last result and by Proposition \ref{pippis2} we have (consistently to the general theory for Lie algebras developed in \cite{PDP}):
\begin{corollary}
\begin{enumerate}
\item The $2n$ elements $Y,Y^2,Y^5,Y^6,\ldots, Y^{4n-3}, Y^{4n-2} $ are a basis of\\
$B^+$  on the exterior algebra generated by the $n-1$ elements  $T_i,\ i=0,\ldots,n-2$. 
\item The $2n$ elements $1,Y^3,Y^4,Y^7,\ldots, Y^{4n-4}, Y^{4n-1}$ are a basis of $B^-$  on the exterior algebra generated by the $n-1$ elements  $T_i,\ i=0,\ldots,n-2$
\end{enumerate}
\end{corollary}

\subsection{The odd orthogonal case}
In the last part of this section we want to generalize our considerations to the case in which the group $G$ considered is not $Sp(2n)$, but $O(2n+1)$. Results are very similar and, almost always, the proofs are the same as in the symplectic case, so we leave details to the reader.

In this case we start to investigate the space $(\bigwedge (M_{2n+1}^{\pm})^*\otimes M_{2n+1})^G$, where $M_{2n+1}^{\pm}$ indicates the space of the symmetric or skew-symmetric matrices with respect to the usual transposition.
\subsubsection{Dimension}
We start analyzing the skew-symmetric case. We recall the isomorphism $M_{2n+1}^-\simeq \bigwedge^2 V$, where $V=\mathbb{C}^{2n+1}$. We have the usual decomposition:
$$\bigwedge\left(\bigwedge^2V\right) =\bigoplus H_{a_1,a_2,\ldots,a_k}^-(V), $$
with $2n+1>a_1>\cdots a_k>0$. The only difference with the symplectic case regards the invariants. This time, there is the invariant only if the diagram has even columns (see \cite{Pr2}).

We have:
\begin{proposition}\label{aldo}
The dimension of the space of invariants $\bigwedge\left(M_{2n+1}^-\right)^G$ is $2^n$.
\begin{proof}
We want that the diagram has even columns. So the sequence $a_1>a_2>...>a_k$ must be of type $2n+1>b_1>b_1-1>b_2>b_2-1>...$, with $b_i$ even (not zero). It follows that the number of these sequences is the number of different sequences composed by even numbers minor than $2n+1$, therefore likewise the symplectic case we have our claim.
\end{proof}
\end{proposition}
The symmetric orthogonal case is very simple by previous considerations. We have $M_{2n+1}^+\simeq S^2(V)$. So by the decomposition \eqref{decsim}, to compute the dimension of invariants we have to consider sequences of type  $2n+1>b_1>b_1-1  \ldots >b_k>b_k-1> 0$, with each $b_i$ even or of type $2n+1>b_1>b_1-1  \ldots b_{k-1}>b_{k-1}-1>b_k=0$. So we have:
\begin{proposition}\label{dimult}
The dimension of the space of invariants $\left(\bigwedge M_{2n+1}^+\right)^G$ is $2^{n+1}$.
\end{proposition}

Furthermore we can compute the dimension of covariants.
\begin{proposition}
The dimension of the space $(\bigwedge (M_{2n+1}^-)^*\otimes M_{2n+1})^G$ (resp. $(\bigwedge (M_{2n+1}^+)^*\otimes M_{2n+1})^G$) is $n2^{n+1}$ (resp. $(2n+1)2^{n+1}$).
\begin{proof}
We will prove only the case $M_{2n+1}^-$, the other case is similar. The proof follows easily from Theorem \ref{ladim}. 

In fact we can reinterpret the orthogonal case from the point of view of the symplectic case. By transposing diagrams we are in the situation of the skewsymmetric symplectic case for $2n$ up to: 
\begin{enumerate}
\item subtract the contribution given by the case in which the first two columns have maximal length and we add two boxes to these. 
\item add the contribution given by the case in which the first row has maximal length and we add two boxes to this.
\end{enumerate}
So we have that the dimension is $n2^{n+1}-2^{n-1}+2^{n-1}=n2^{n+1}$.

\end{proof}
\end{proposition}

\subsection{Structure}
We proceed to describe explicitly of the space $B:=(\bigwedge (M_{2n+1}^-)^*\otimes M_{2n+1})^G$. We can see such algebra as a left module on the algebra $A:=(\bigwedge (M_{2n+1}^-)^*)^G$ which is, as an algebra, generated by the elements $Tr(St_{2k+1})$ and similarly to the symplectic case can be described as the exterior algebra.
\begin{proposition}
The algebra $A=(\bigwedge (M_{2n+1}^-)^*)^G$ is the exterior algebra in the elements $T_h=Tr(St_{4h+3})$, with $h=0,1,...,n-1$.
\end{proposition}
\begin{proof}
By the proof of Proposition \ref{pippis2} and the Amitsur-Levitzki's Theorem we have that the elements $T_h$, $h=0,...,n-1$ generate the invariants.
Since the dimension is exactly $2^n$, we have our claim.
\end{proof}

To study covariants let us recall the by classical invariant theory we have that the space $B=(\bigwedge (M_{2n+1}^-)^*\otimes M_{2n+1})^G$ is generated as a module on $A=(\bigwedge (M_{2n+1}^-)^*)^G$ by the elements $Y^i$. So reasoning as in \ref{tony} and \ref{tony2} we can state the following.
\begin{teorem}\label{tony4}
\begin{enumerate}
 \item  As a algebra, $B$ is generated by $A$ and the element $Y$.
 \item   $B$ is a free module on the 
  the exterior algebra $A_{n-1}\subset A$ generated by $T_0,...,T_{n-2}$, with basis $1,Y,...,Y^{4n-1}$.
  
  With notations of Section 2 we can consider the canonical surjective homomorphism $\pi :A[t]\rightarrow B$. We have:
 \item  The Kernel of the homomorphism $\pi$ is the principal ideal generated by the element
 \begin{equation}\label{oddhutch}
 (2n+1)t^{4n-1}-\sum_{i=0}^{n-1} T_{n-i-1}\wedge t^{4i}, \end{equation}
 \end{enumerate}
 
  \end{teorem}

Remark that multiplying by $t$ the generator in \eqref{oddhutch}, we obtain an identity between an element of even degree $t^{4n}$ and elements of odd degree $t^{4i+1}$. By the relation $T_{n-i-1}\wedge t^{4i+1}=-t^{4i+1}\wedge T_{n-i-1}$, we can deduce $t^{4n}=0$. We have recovered a well known result of Hutchinson (see \cite{Hu}):

\begin{proposition}
The skew-symmetric standard polynomial $St_{4n}(x_1,...,x_{4n})$ is zero on the space $M_{2n+1}^-$.
\end{proposition}
From the Proposition \ref{pippis2} we also get: 

\begin{corollary}\begin{enumerate}
\item The algebra $(\bigwedge (M_{2n+1}^-)^*\otimes M_{2n+1}^-)^G$ is a free module on the exterior algebra $A_{n-1}^-$ with basis $Y,Y^2,Y^5,Y^6,...,Y^{4n-3},Y^{4n-2}$. 
\item The algebra $(\bigwedge (M_{2n+1}^-)^*\otimes M_{2n+1}^+)^G$ is a free module on the exterior algebra $A_{n-1}^-$ with basis $1,Y^3,Y^4,Y^7,...,Y^{4n-4},Y^{4n-1}$.
\end{enumerate}
\end{corollary}

Let us pass to the symmetric case. We can describe easily the structure of invariants. By classical invariant theory we have that this algebra is generated by the elements $Tr(St_{2k+1})$, but by Propositions \ref{pippis}, \ref{dimult} and Amitsur-Levitzki we have:
\begin{lemma}
The algebra of invariants $A:=\left(\bigwedge (M_{2n+1}^+)^* \right)^G$ is the exterior algebra in the elements $T_0,...,T_n$ ($T_i:=Tr(St_{4i+1})$).
\end{lemma}
So we can describe the covariants $B:=(\bigwedge (M_{2n+1}^+)^*\otimes M_{2n+1})^G$.
\begin{teorem}\label{tony5}
\begin{enumerate}
 \item  As a algebra $B$ is generated by $A$ and the element $X$.
 \item   $B$ is a free module on the 
  the exterior algebra $A_{n-1}\subset A$ generated by $T_0,...,T_{n-1}$, with basis $1,X,...,X^{4n+1}$.
  
  With notations of Section 2 we can consider the canonical surjective homomorphism $\pi :A[t]\rightarrow B$. We have:
 \item  The Kernel of the homomorphism $\pi$ is the principal ideal generated by the element
 \begin{equation}\label{ultid}
 (2n+2)t^{4n+1}-\sum_{i=0}^{n} T_{n-i}\wedge t^{4i+2}, \end{equation}
\end{enumerate}

\begin{proof}
We need to prove only the third part. By the analog of the general formula \eqref{genmat} we have
$$
1\wedge T(4n+1)=T(4n+1)=-\sum_{i=1}^{2n+1}Z^{2i} \wedge  T(2(2n+2-i)-1) +(2n+2)Z^{4n+1},
$$
so we deduce $i$ is odd and we have:
$$T_n=-\sum_{i=0}^{n}X^{4i+2}\wedge T_{n-i}+(2n+2)X^{4n+1},$$
which is our claim.
\end{proof}

\end{teorem}

\begin{corollary}
\begin{enumerate}
\item The algebra $(\bigwedge (M_{2n+1}^+)^*\otimes M_{2n+1}^-)^G$ is a free left module on the exterior algebra $A_{n-1}$, with basis $X^2,X^3,...,X^{4n-2},X^{4n-1}$. 
\item The algebra $(\bigwedge (M_{2n+1}^+)^*\otimes M_{2n+1}^+)^G$ is a free left module on the exterior algebra $A_{n-1}$, with basis $1,X,...,X^{4n},X^{4n+1}$.
\end{enumerate}
\end{corollary}

\section{Nonclassical cases}
\subsection{The case $\mathfrak{so}(2n)/\mathfrak{so}(2n-1)$}
In this part of the work we want to investigate the remaining cases which we have discussed in the section 1.2.3. In this case $\mathfrak{g}=\mathfrak{k}\oplus \mathfrak{p}$, where $\mathfrak{g}=\mathfrak{so}(2n),\ \mathfrak{k}=\mathfrak{so}(2n-1)$ and $\mathfrak{p}=\mathbb{C}^{2n-1}$. Let us start considering the invariants of the the space 
$$ \bigwedge \left( \frac{\mathfrak{so}(2n)}{\mathfrak{so}(2n-1)} \right)^*, $$
under di action of $\mathfrak{so}(2n-1)$ (acting by derivation), we will denote this space by $A$. If we denote $V:=\mathbb{C}^{2n-1}$ we have that the action of $\mathfrak{so}(2n-1)$ is the natural action on $\wedge V$.

From the point of view of the classical invariant theory, we have to look at multilinear skew-symmetric invariants under simultaneous conjugation of the group $K=SO(2n-1)$. By the FFT we deduce they are the exterior algebra in the element $p=[v_1,...,v_{2n-1}]$, i.e., the determinant.

Let us investigate the space of covariants. We have:
$$ \left(\bigwedge\mathfrak{p}^*\otimes \mathfrak{k}\right)^{K}\simeq \left(\bigwedge (V)^*\otimes \bigwedge^2 V\right)^K, $$
we can deduce easily that this space has dimension 2. 
Then we can consider in $\left(\bigwedge^{2n-3} (V)\otimes \bigwedge^2 V\right)^*$ the unique (up to scalar) element $p=[v_1,...,v_{2n-1}]$ and we take the corresponding covariant element:
$$\omega_1=\sum [v_1,...v_{2n-3},e_i,e_j]e_i\wedge e_j\in \left(\bigwedge^{2n-3} (V)^*\otimes \bigwedge^2 V\right)^K.$$
On the other hand if we consider in $\left(\bigwedge^{2} (V)\otimes \bigwedge^2 V\right)^*$
the element $([x_1,x_2],[x_3,x_4])$, where $x_i\in \mathfrak{so}(2n)$ we have the corresponding element $\omega_2\in \left(\bigwedge^{2} (V)^*\otimes \bigwedge^2 V\right)^K.$

So we have:
\begin{teorem}
The space $B^+:=\left(\bigwedge\mathfrak{p}\otimes\mathfrak{k}\right)^K$ is the vector space of dimension $2$ with basis (over $\mathbb{C}$) $\omega_1,\omega_2$.
\end{teorem}
We can do similar considerations for the case $B^-:=\left(\bigwedge \mathfrak{p}\otimes \mathfrak{p}\right)^K$. From the point of view of the classical invariant theory we have to look at the space 
$$\left(\bigwedge (V)^*\otimes V\right)^K.$$

As before we have that this space has dimension $2$ and a basis is given respectively by the covariant element $\theta_1$ corresponding to $[v_1,...,v_{2n-1}]\in \left(\bigwedge^{2n-2} (V)\otimes V\right)^*$ and the covariant $\theta_2$ corresponding to $(v_1,v_2)\in \left(\bigwedge^1 (V)\otimes V\right)^*$. So we have:
\begin{teorem}
The space $B^-:=\left(\bigwedge\mathfrak{p}\otimes\mathfrak{p}\right)^K$ is the vector space of dimension $2$ with basis (over $\mathbb{C}$) $\theta_1,\theta_2$.
\end{teorem}

\subsection{Exceptional case}
The last case we consider is the symmetric space $E_6/F_4$. Let us start recalling some general results for Lie algebras (see for example \cite{HoSe}, \cite{Koszul}).

Let $\g$ be a finite dimensional Lie algebra over $\mathbb{C}$ and let $\mathfrak{u}\subset\g$ a subalgebra reductive in $\g$, i.e. the adjoint representation $\mathfrak{u}\times \g\rightarrow \g$ is completely reducible. We denote by $\mathfrak{p}$ a stable complement under the adjoint action, so $\g=\mathfrak{u}\oplus\mathfrak{p}$ and as a vector space $\mathfrak{p}\simeq\g/\mathfrak{u}$. In this setting we can state the following (see \cite{HoSe} Theorem 12):
\begin{teorem}\label{Hoch}
Let $\g, \mathfrak{u}$ as before, and assume furthermore that the restriction homomorphism maps $\left(\bigwedge\g\right)^{\g}$ onto $\left(\bigwedge\mathfrak{u}\right)^{\mathfrak{u}}$. Then we have 
$$\left(\bigwedge  \mathfrak u\right)^{\mathfrak{u}}\otimes \left(\bigwedge \mathfrak p\right)^{\mathfrak{u}}\simeq \left(\bigwedge  \mathfrak g\right)^{\mathfrak{g}}.$$
\end{teorem}
In our case we have that $\g=Lie(E_6)$ and $\mathfrak{u}=Lie(F_4)$. If we denote by $N$ the algebra of invariants $\left(\bigwedge \g\right)^{\g}=\left(\bigwedge \g\right)^{G}$, we know by classical results that $N=\wedge (p_3,p_9,p_{11},p_{15},p_{17},p_{23})$, where $p_i$ is a primitive invariant of degree $i$. While $\left(\bigwedge \mathfrak{u}\right)^U=\wedge (q_3,q_{11},q_{15},q_{23})$, with $q_i=\pi_+ (p_i)$, where $\pi_+:\left(\bigwedge\g\right)^G\rightarrow \left(\bigwedge\mathfrak{u}\right)^U$ is the natural map induced by $p_+:\g\rightarrow \mathfrak{u}$.
So by the Theorem \ref{Hoch} we have that $A:=\left(\bigwedge \mathfrak{p}\right)^U$ is the exterior algebra $\wedge (r_9,r_{17})$, where $r_i=\pi_-(p_i)$ and $\pi_-:N\rightarrow A$ is induced by $p_-:\g\rightarrow \mathfrak{p}$.

Now we can consider $M=hom(\bigwedge \g,\g)^G$ and $B^+=hom(\bigwedge \mathfrak{p},\mathfrak{u})^U$. Further we can decompose in a natural way $M=M^-\oplus M^+$ as
$$M=\left(\bigwedge \g\otimes \mathfrak{p}\right)^G\oplus \left(\bigwedge \g\otimes \mathfrak{u}\right)^G,$$
so we can define an homorphism $\Gamma:M\rightarrow B^+$ such that $\phi\mapsto \pi_-(\phi^+)$. We recall (See \cite{PDP}) that $M$ is a free module on the subalgebra $\wedge (p_3,p_9,p_{11},p_{15},p_{17})\subset N$ with basis certain elements $u_{i-2},f_{i-1}$, where $i$ runs over the degrees of the generators of invariants and $u_j,f_j$ has degree $j$.
Using the software ''Lie'' we have the following table that describes the graduated module $B^+$:
\begin{equation}\label{tabella1}
\begin{tabular}{r|c|c|c|c|c|c|c|c|c|c|c|c|c|c|c|}
degree&0&1&2&3&4&5&6&7&8&9&10&11&12&13\\ \hline
dimension&0&0&1&0&0&0&0&1&0&0&1&1&0&0\\ \hline
\end{tabular}
\end{equation}

Remark that the dimension of $\mathfrak{p}$ is $26$, so by the Poincar\'e duality the previous table it is sufficient to describe the all space.
We deduce $dim(B^+)=8$ and each covariant has different degree.

Let us consider $g_i:=\Gamma(f_i)$, with $i=2,10$ and $v_i:=\Gamma(u_i)$, with $i=7,15$. We want to prove that $B^+$ is a free module on the one-dimensional algebra generated by $r_9$ with basis $g_2,g_{10},v_7,v_{15}$. So we only need to prove that all $g_i,v_i$ and their products by $r_9$ is nonzero.

We denote by $(\cdot,\cdot)_{\g}$ the Killing form of $\g$, we know that it induces, by restriction, a non degenerate invariant form on $\mathfrak{u}$, which will be denoted by $(\cdot,\cdot)_{\mathfrak{u}}$. Our goal is to prove that, up to a non zero scalar, 
\begin{equation}\label{ilprod}(g_2,v_{15})_{\mathfrak{u}}=(v_7,g_{10})_{\mathfrak{u}}=r_{17},\end{equation}
so our claim will follow. Let us start considering the first scalar product $(v_7,g_{10})_{\mathfrak{u}}$. We know (see \cite{PDP}) that, up to a non zero scalar, we have:
$$(u_7,f_{10})_{\mathfrak{g}}=p_{17}.$$
Let us write $u_7=u_7^++u_7^-$. Using the software ''Lie'' we have the following table describing $B^-:=\left(\bigwedge\mathfrak{p}\otimes \mathfrak{p}\right)^H$: 
\begin{equation}\label{tabella2}
\begin{tabular}{r|c|c|c|c|c|c|c|c|c|c|c|c|c|c|c|}
degree   &0&1&2&3&4&5&6&7&8&9&10&11&12&13\\ \hline
dimension&0&1&0&0&0&0&0&0&1&1&1&0&0&0\\ \hline
\end{tabular}
\end{equation}
We deduce $\pi_-(u_7^-)=0.$ On the other hand, we have
$$
\begin{aligned}p_{17}&=(u_7,f_{10})_{\g}\\
                    &=(u_7^++u_7^-,f_{10}^++f_{10}^-)_{\g}\\
                    &=(u_7^+,f_{10}^+)_{\g}+(u_7^-,f_{10}^-)_{\g},
\end{aligned}
$$
so applying $\pi_-$ we have
$$
\begin{aligned}r_{17}&=\pi_-(p_{17})\\
                   &=\pi_-((u_7,f_{10})_{\g})\\
                   &=(\pi_-(u_7^+),\pi_-(f_{10}^+))_{\mathfrak{h}}\\
                   &=(v_7,g_{10})_{\mathfrak{h}}.
\end{aligned}
$$
We can do the same considerations for $(g_2,v_{15})_{\mathfrak{u}}$. Furthermore, by linearity we have:
$$(r_9g_2,v_{15})_{\mathfrak{u}}=(r_9v_7,g_{10})_{\mathfrak{u}}=(g_2,r_9v_{15})_{\mathfrak{u}}=-(v_7,r_9g_{10})_{\mathfrak{u}}=r_9r_{17}\neq 0.$$

We can do similar considerations for the space $B^-$. Let us consider $v_i:=\pi_-(u_i^-),\ i=1,9$ and $g_i:=\pi_-(f_i^-),\ i=8,16$. By the table \eqref{tabella1} we deduce $\pi^-(u_i^+)=0,\ i=1,9$. As before, the Killing form $(\cdot,\cdot)_{\g}$ induces a non degenerate invariant form on $\mathfrak{p}$, we denote it by $(\cdot,\cdot)_{\mathfrak{p}}$ and we have, up to a nonzero scalar, that
$$(v_1,g_{16})_{\mathfrak{p}}=(g_8,u_9)_{\mathfrak{p}}=r_{17}, $$
and that
$$(r_9g_8,v_{9})_{\mathfrak{p}}=(r_9v_1,g_{16})_{\mathfrak{p}}=(g_8,r_9v_{9})_{\mathfrak{p}}=-(v_1,r_9g_{16})_{\mathfrak{p}}=r_9r_{17}\neq 0.$$
We can summarize our considerations in the following:
\begin{teorem}\begin{enumerate}
\item The space $B^+$, of dimension $2\cdot 2^2$, is a free module on the subalgebra $\wedge(r_9)\subset A$, with basis $g_2,g_{10},v_7,v_{15}$.
\item The space $B^-$, of dimension $2\cdot 2^2$, is a free module on the subalgebra $\wedge(r_9)\subset A$, with basis $g_8,g_{16},v_1,v_{9}$.
\end{enumerate}
\end{teorem}
\section{Conclusions}
We now summarize the previous results and deduce Theorem \ref{mainthe}. Let us start analyzing our results about spaces of type $B^+$. 

Let $\g$ be a simple complex Lie algebra, $\sigma:\g\rightarrow \g$ an indecomposable involution and $\g\simeq \mathfrak{p}\oplus\mathfrak{k}$ the Cartan decoposition with $\mathfrak{k}$ the fixed point set of $\sigma$. Let us assume that $\left(\bigwedge\mathfrak{p}^*\right)^{\mathfrak{k}}$ is an exterior algebra of type $\wedge (x_1,...,x_r)$, where the $x_i$'s are ordered by their degree and $r$ is such that $r:=rk(\mathfrak{g})-rk(\mathfrak{k})$. Then we have:
\begin{teorem}
The algebra $B^+:=\left(\bigwedge \mathfrak{p}^*\otimes \mathfrak{k}\right)^\mathfrak{k}$ is a free module of rank $2r$ on the subalgebra of $\left(\bigwedge \mathfrak{p}^*\right)^\mathfrak{k}$ generated by the elements $x_1,...,x_{r-1}$. 
\end{teorem}
\begin{proof}
The only thing we have to prove is that our results about the orthogonal and symplectic groups can be seen from this point of view in a way that implies our theorem. The case $B^+$ is particularly simple. In fact by little changes in the proof of the calculus of the dimensions we see that $dim(B^+)=(n-1)2^{n-1}$ for the symplectic case and $n2^{n}$ for the odd orthogonal case. Then our claim follows choosing as a basis the elements $Y^i:=X^i-\frac{1}{m}Tr(X^i)$, with $m=2n,2n+1$ and from the fact that for traceless matrices $T_0=0.$ 
\end{proof}

Further we have the same result for $B^-$. We just have to remark that $1\notin B^-$ and in both symplectic and orthogonal cases we can obtain the generator of higher degree, respectively $X^{4n-3}-\frac{1}{2n}Tr(X^{4n-3})$ and $X^{4n+1}-\frac{1}{2n+1}Tr(X^{4n+1})$ by the characterizing identities we have proved. We deduce:
\begin{teorem}
The algebra $B^-:=\left(\bigwedge \mathfrak{p}^*\otimes \mathfrak{p}\right)^\mathfrak{k}$ is a free module of rank $2r$ on the subalgebra of $\left(\bigwedge \mathfrak{p}^*\right)^\mathfrak{k}$ generated by the elements $x_1,...,x_{r-1}$. 
\end{teorem}

\end{document}